\documentclass[reqno]{amsart}
\usepackage{amscd}

\theoremstyle{plain}
  \newtheorem{theo}{Theorem}
  \newtheorem{coro}[theo]{Corollary}
  \newtheorem{prop}[theo]{Proposition}
  \newtheorem{lemm}[theo]{Lemma}
  \newtheorem*{quest*}{Question}  

\theoremstyle{remark}
  \newtheorem*{rema*}{Remark}
  \newtheorem{rema}[theo]{Remark}

\newcommand\ie{i.e.\ }

\renewcommand\o{\circ}
\newcommand \al{\alpha}
\newcommand\be{\beta}
\newcommand\ga{\gamma}
\newcommand\de{\delta}

\newcommand\ze{\zeta}
\newcommand\et{\eta}

\newcommand\si{\sigma}

\newcommand\ph{\varphi}

\newcommand\om{\omega}
\newcommand\Ga{\Gamma}

\newcommand\La{\Lambda}

\newcommand\Om{\Omega}

\newcommand\Ad{\text{\rm Ad}}

\newcommand\Aut{\text{\rm Aut}}

\newcommand\Diff{\text{\rm Diff}}

\renewcommand\b{\text{\rm b}}
\newcommand\Ker{\text{\rm Ker}}

\newcommand\vol{\mu}
\newcommand\OGr{\text{\rm Gr}}

\newcommand\Ham{\text{\rm Ham}}

\newcommand\per{\text{\rm per}}

\newcommand\vf{\mathfrak X}

\newcommand\Diffvol{\text{\rm Diff}(M,\vol)}

\newcommand\g{\mathfrak g}

\newcommand\M{\mathcal M}

\newcommand\X{\mathfrak X}

\newcommand\x{\times}

\newcommand\R{\mathbb R}
\newcommand\Z{\mathbb Z}

\begin{document}

\title
{Lichnerowicz cocycles and central Lie group extensions} 

\author{Cornelia Vizman}

\address{Cornelia Vizman, 
         West University of Timi\c soara, Department of Mathematics, 
         Bd. V.Parvan 4, 300223--Timi\c soara, Romania.}

\email{vizman@math.uvt.ro}


\keywords{Lichnerowicz cocycle, central extension, coadjoint orbit}

\subjclass[2000]{58B20}

\begin{abstract}
We present a geometric construction of central extensions of 
covering groups of the group of volume preserving diffeomorphisms, 
integrating central extensions of the Lie algebra of divergence free vector fields
defined by Lichnerowicz cocycles.
Certain covering spaces of non--linear Grassmannians can be realized as prequantizable
coadjoint orbits in these Lie group extensions.
\end{abstract}
\dedicatory
{Dedicated to Professor Dan Papuc 
on the occasion of his 80th anniversary}
\maketitle

\section{Introduction}

There is a geometric construction of central Lie group extensions
as pull-back of the prequantization central extension 
involving the quantomorphism group.
The ingredients are a Lie group $G$ with Lie algebra $\g$, a connected prequantizable symplectic 
manifold $(\mathcal M,\Om)$
and a Hamiltonian action of $G$ on $\mathcal M$.
Let $\mathcal P\to\mathcal M$ be the principal $S^1$--bundle with connection 
1--form $\et$ and curvature $\Om$. We denote by $\Aut (\mathcal P,\et)$ 
the group of quantomorphisms, \ie
the connected component of the group of equivariant connection
preserving diffeomorphisms of $\mathcal P$, and by
$\Ham(\mathcal M,\Om)$ the group of Hamiltonian diffeomorphisms of
$\mathcal M$.
The prequantization central extension \cite{Ko70} \cite{S70} associated to $(\mathcal M,\Om)$
is
\begin{equation}\label{kostant}
1\to S^1\to\Aut(\mathcal P,\eta)\to\Ham(\mathcal M,\Om)\to 1.
\end{equation}
For a finite dimensional manifold $\mathcal M$, this is an extension of Lie groups \cite{RS}.
The pull--back of the prequantization extension \eqref{kostant} to $G$ leads to a 1--dimensional central Lie group extension
of $G$. This is true even if $\mathcal M$ is infinite dimensional \cite{NV03}.
The corresponding Lie algebra extension of $\g$ is defined by the Lie algebra 2--cocycle $(X,Y)\mapsto -\Om(\ze_X,\ze_Y)(x_0)$,
where $x_0\in\M$ and $\ze_X$ is the fundamental vector field on $\M$ for $X\in\g$.

Let $M$ be a compact $m$--dimensional manifold with integral volume form 
$\mu$ and let $\OGr_{m-2}(M)$ denote the non--linear Grassmannian 
consisting of all oriented codimension two submanifolds of $M$.
There is a natural symplectic form on $\OGr_{m-2}(M)$ \cite{I96},
denoted by $\tilde\mu$ since it is obtained from the volume form $\mu$ by the tilda map \cite{HV04}. 
It is the higher dimensional version of the natural symplectic form on the space of knots in 
$\R^3$ \cite{MW83}.

The group of exact
volume preserving diffeomorphisms of $M$ acts in a Hamiltonian way on 
$\OGr_{m-2}(M)$. 
The symplectic manifold $\OGr_{m-2}(M)$ is prequantizable if the volume form $\mu$ is integral.
Then the pull--back of the central extension (\ref{kostant}) associated to a connected component $\mathcal M$ of $\OGr_{m-2}(M)$ with symplectic form $\Om=\tilde\mu$ is a 1--dimensional central Lie group extension of the group of exact 
volume preserving diffeomorphisms \cite{I96} \cite{HV04}. 
For any codimension 2 submanifold $N_0$ of $M$ in $\M$, 
a Lie algebra 2-cocycle describing the corresponding central Lie algebra extension of 
the Lie algebra of exact volume preserving vector fields is
$\om_{N_0}(X,Y)=\int_{N_0}i_Xi_Y\mu$.
This cocycle is cohomologous to the Lichnerowicz cocycle
$\om_{\et}(X,Y)=\int_M\et(X,Y)\mu$,
where $\et$ is a closed 2--form on $M$ such that $[\et]\in H^2(M,\R)$ and $[N_0]\in H_{m-2}(M,\R)$ are Poincar\'e dual.
In particular $[\et]$ is an integral cohomology class.

The same formulas as above define cocycles $\om_\et$ and $\om_{N_0}$
on the Lie algebra of divergence free vector fields,
but the corresponding (isomorphic) central Lie algebra extensions
are in general not integrable to the group of volume preserving diffeomorphisms of $M$.
The integrability to the universal covering group
$\widetilde\Diff(M,\mu)_0$ of its identity component is shown in \cite{N04} by computing
the period group of the cocycle $\om_\et$. A construction of this extension
with the help of the path group of $\Diff(M,\mu)$ is presented in \cite{V08}.

In this article we obtain the central extension of $\widetilde\Diff(M,\mu)_0$ integrating $\om_{N_0}$
by pulling back the prequantization central extension
with the help of the canonical Hamiltonian action of $\widetilde\Diff(M,\mu)$ on the symplectic manifold $\widetilde{\mathcal M}$, the universal cover of $({\mathcal M},\tilde\mu)$.
Further, with the help of the flux homomorphism of $\om_{N_0}$,
we determine the smallest covering group $\overline\Diff(M,\mu)$  of $\Diff(M,\mu)_0$
on which $\om_{N_0}$ is integrable. For infinite dimensional groups,
the flux homomorphism of a Lie algebra 2-cocycle is an obstruction to its integrability \cite{N02}.

Using the momentum map, $\mathcal M$ 
can be realized as a coadjoint orbit of the central Lie group extension of the group of exact volume preserving diffeomorphisms integrating $\om_{N_0}$ \cite{HV04}.
It turns out that the covering space 
$$
\bar{\mathcal M}=\widetilde{\mathcal M}/\Ker(\pi_1({\mathcal M})\to H_{m-1}(M,\R))
$$  
of ${\mathcal M}$  is a prequantizable coadjoint orbit of the central Lie group extensions 
of $\widetilde\Diff(M,\mu)$ and $\overline\Diff(M,\mu)$.

The plan of the present paper is as follows.
In Section 2 we consider 2--cocycles on the Lie algebra of divergence free vector fields,
cohomologous to the Lichnerowicz cocycle.
In Section 3 we present what is known about the integrability of these cocycles. 
With the help of the flux homomorphism 
we determine the minimal covering group of $\Diff(M,\mu)_0$ on which 
these cocycles are integrable. 
In Section 4 we list some properties of non--linear Grassmannians. 
Geometric constructions, via the prequantization extension, 
of the Lie group extensions whose existence was shown in Section 3 
are presented in Section 5,
together with the special coadjoint orbit $\bar{\mathcal M}$.


\section{2--cocycles on the Lie algebra of divergence free 
vector fields}

Let $M$ be a compact $m$--dimensional manifold with volume form $\mu$. 
The infinitesimal flux homomorphism is defined on the Lie algebra 
$\X(M,\mu)$ of divergence free vector fields by:
$$
s_\mu:\X(M,\mu)\to H^{m-1}(M,\R),\quad s_\mu(X)=[i_X\mu].
$$
The Lie algebra of exact divergence free vector fields 
${\X}_{ex}(M,\mu)$ is the kernel of $s_\mu$, hence an ideal of 
$\X(M,\mu)$.

The corresponding Lie groups are the group of volume preserving
diffeomorphisms $\Diff(M,\mu)$ with the subgroup of exact volume 
preserving diffeomorphism $\Diff_{ex}(M,\mu)$. The last one 
is the kernel of the flux homomorphism $S_\mu$, due to Thurston, 
integrating the infinitesimal flux homomorphism
$s_\mu$ \cite{B}. For a Lie group structure on diffeomorphism groups see \cite{KM97}.

Every closed 2--form $\et$ on $M$ defines a Lie algebra 2--cocycle on 
${\X}(M,\mu)$, called Lichnerowicz cocycle: 
$$
\om_\et(X,Y)=\int_M\et(X,Y)\mu.
$$
The cocycle condition is easily verified, taking into account that the vector fields are divergence free:
\begin{equation*}
\sum_{cycl}\om_\et([X,Y],Z)=\sum_{cycl}\int_M\et([X,Y],Z)\mu
=\sum_{cycl}\int_ML_X\et(Y,Z)\mu=0.
\end{equation*}

\begin{rema}
The map $\et\mapsto\om_\et$ induces an isomorphism between $H^2(M,\R)$ and $H_c^2(\X_{ex}(M,\mu))$, the second continuous cohomology space of the Lie algebra of exact divergence free vector fields \cite{R95}.

In the exact sequence of Lie algebras
\begin{equation}\label{sn}
0\to\X_{ex}(M,\mu)\to\X(M,\mu)\stackrel{s_\mu}{\to} H^{m-1}(M,\R)\to 0,
\end{equation}
the ideal $\X_{ex}(M,\mu)$ of 
exact divergence free vector fields
is perfect \cite{Li74}. Therefore the pull--back by $s_\mu$
is an injective homomorphism in continuous Lie algebra cohomology,
$s_\mu^*:H_c^2(H^{m-1}(M,\R))=\La^2H^{m-1}(M,\R)^*\to H^2_c(\X(M,\mu))$.

In conclusion the continuous cohomology space $H_c^2(\X(M,\mu))$
of the Lie algebra of divergence free vector fields is isomorphic to
$H^2(M,\R)\oplus\La^2H^{m-1}(M,\R)^*$. 
\end{rema}

Another Lie algebra 2--cocycle on the Lie algebra of divergence free vector fields 
is defined by
\begin{equation*}
\om_{N_0}(X,Y)=\int_{N_0}i_Xi_Y\mu,
\end{equation*} 
where $N_0$ is an oriented compact codimension 2 submanifold of $M$. 
In the same way, every $(m-2)$--cycle $c$ on $M$ defines a Lie algebra 2--cocycle on $\X(M,\mu)$:
$$
\om_c(X,Y)=\int_c i_Xi_Y\mu.
$$

\begin{prop}
If the cohomology class  of the closed 2--form $\et$ on $M$ is Poincar\'e dual to $[c]\in H_{m-2}(M,\R)$,
then $\om_c$ and the Lichnerowicz cocycle
$\om_\et$ are cohomologous Lie algebra 2-cocycles on the Lie algebra of divergence free vector fields.
\end{prop}

\begin{proof}
Every continuous linear projection $P:\Om^{m-2}(M)\to Z^{m-2}(M)$
on the space of closed $(m-2)$--forms is of the form $P=I-\b d$ with 
$\b:B^{m-1}(M)\to\Om^{m-2}(M)$ a continuous linear  right inverse to 
the differential $d$, where $B^{m-1}(M)=d\Om^{m-2}(M)$. Then another Lie algebra 2--cocycle on ${\X}(M,\mu)$
is given by
$$
\om_{c,P}(X,Y)=\int_c Pi_Xi_Y\mu.
$$

We show that both cocycles $\om_\et$ and $\om_c$ are cohomologous to $\om_{c,P}$: 
\begin{align*}
\om_\et(X,Y)&=\int_Mi_Yi_X\et\wedge\mu=\int_M\et\wedge i_Xi_Y\mu
=\int_M\et\wedge Pi_Xi_Y\mu+\int_M\et\wedge \b di_Xi_Y\mu\\
&=\om_{c,P}(X,Y)+\int_M\et\wedge \b i_{[X,Y]}\mu
\end{align*}
and
\begin{equation*}
\om_c(X,Y)-\om_{c,P}(X,Y)=\int_c\b di_Xi_Y\mu=\int_c\b i_{[X,Y]}\mu.
\end{equation*} 
A consequence of these computations is that $\om_c$ and $\om_{c,P}$ are Lie algebra 2--cocycles on the Lie algebra of divergence free vector fields.
\end{proof}

\begin{rema}
The Lie algebra extensions by $\om_\et$, $\om_c$ and $\om_{c,P}$ of $\X(M,\mu)$ are isomorphic, 
but each of these cocycles has its advantage. 
When a corresponding central Lie group extension exists,
the first one allows to find the geodesic equation 
for the right invariant $L^2$--metric. One obtains the superconductivity equation \cite{AK} \cite{V01}.
The second cocycle describes the Lie algebra extensions corresponding to Lie group extensions 
constructed via the prequantization extension, as we will see below.
The third cocycle is just the pairing with an $(m-2)$--cycle  
$c$ on $M$ of the universal Lichnerowicz cocycle of ${\X}_{ex}(M,\mu)$ \cite{R95}:
\begin{equation}\label{univ}
\om:{\X}_{ex}(M,\mu)\x{\X}_{ex}(M,\mu)\to H^{m-2}(M,\R),\quad
\om(X,Y)=[Pi_Xi_Y\mu].
\end{equation}
Indeed, the space of exact divergence free vector fields can be identified with the space $B^{m-1}(M)$
of exact $(m-1)$--forms on $M$,
so we have the following exact sequence of vector spaces
\begin{equation}\label{rog}
0\to H^{m-2}(M,\R)\to \Om^{m-2}(M)/B^{m-2}(M)\to{\X}_{ex}(M,\mu)\to 0.
\end{equation}
The Lie algebra bracket on $\Om^{m-2}(M)/B^{m-2}(M)$:
$$
[\al+B^{m-2}(M),\be+B^{m-2}(M)]=i_{X_\al}i_{X_\be}\mu+B^{m-2}(M),
$$
where $i_{X_\al}\mu=d\al$, makes (\ref{rog})
into an exact sequence of Lie algebras.
The Lie algebra cocycle defining this central extension of ${\X}_{ex}(M,\mu)$ by $H^{m-2}(M,\R)$ can be calculated with the section 
$X_\al\mapsto \b(d\al)+B^{m-2}(M)$ of (\ref{rog}). Using the fact that $X_{\b(d\al)}=X_\al$, one gets the 2--cocycle 
\eqref{univ}.
\end{rema}


\section{Flux and period homomorphism for the Lichnerowicz cocycle}

Let $\g$ be a topological Lie algebra
and $\om$ a continuous Lie algebra 2--cocycle on $\g$.
The Lie algebra bracket on $\R\x\g$: 
$$
[(a,X),(b,Y)]=(\om(X,Y),[X,Y])
$$
defines the Lie algebra central extension $\R\x_\om\g$ of $\g$ by $\R$. 
Let $G$ be a connected Lie group with Lie algebra $\g$ and 
$\tilde G$ its universal covering group. 

There are two obstructions to the integrability of $[\om]\in H^2_c(\g)$ to a central extension of an infinite dimensional Lie group $G$, \ie to finding a Lie group extension of $G$ integrating the Lie algebra extension $\R\x_\om\g$ of $\g$. 
One obstruction involves the period group and depends on $\pi_2(G)$,
the other one involves the flux homomorphism and depends on $\pi_1(G)$ \cite{N02}.

The infinitesimal flux cocycle $f_\om:X\in\g\mapsto i_X\om\in C_c^1(\g)$
is a 1--cocycle on $\g$ with values in the $\g$--module $C_c^1(\g)$ of 
continuous linear maps from $\g$ to $\R$.
We denote by $X^r$ the right invariant vector field on $G$ defined
by $X\in\g$ and by $\om^l$ the left invariant 2--form on $G$ defined by $\om$.
The abstract flux 1--cocycle 
associated to the Lie algebra cocycle $\om$ is 
\begin{equation*}
\tilde F_\om:\tilde G\to C_c^1(\g),\quad\tilde F_\om([\ga])(X)=-\int_\ga i_{X^r}\om^l.
\end{equation*}
Here $[\ga]\in\tilde G$ denotes the homotopy class of a path $\ga$
in $G$ starting at the identity.
Another expression for the flux 1--cocycle is \cite{N04}
\begin{equation}\label{expr}
\tilde F_\om([\ga])(X)=\int_0^1\om(\ga(t)^{-1}\ga'(t),\Ad(\ga(t))^{-1}X)dt.
\end{equation}
By restricting $\tilde F_\om$ to $\pi_1(G)$ 
we get the flux homomorphism $F_\om:\pi_1(G)\to H_c ^1(\g)$,
where $H_c^1(\g)$ denotes the first continuous cohomology space of the Lie algebra $\g$.

\begin{lemm}\label{lema}
The flux homomorphism associated to the cocycle $\om_{N_0}$ is 
\begin{equation}\label{long}
F_{N_0}:\pi_1(\Diff(M,\mu))\to H_c^1(\X(M,\mu)),\quad
F_{N_0}([\ph_t])(X)=\int_{\ph_{N_0}}i_X\mu,
\end{equation}
where the $(m-1)$--cycle $\ph_{N_0}$ is
$(t,x)\in [0,1]\x N_0\mapsto\ph_t(x)\in M$.
\end{lemm}

\begin{proof}
The adjoint action in $\Diff(M,\mu)_0$ is $\Ad(\ph)X=(\ph^{-1})^*X$
and the relation between the left logarithmic derivative 
$\de^l\ph_t=T\ph_t^{-1}.\frac{d}{dt}\ph_t$ and the right
logarithmic derivative is $\de^r\ph_t=\Ad(\ph_t)\de^l\ph_t=(\ph_t^{-1})^*\de^l\ph_t$. 
Hence 
\begin{align*}
F_{N_0}([\ph_t])(X)&
\stackrel{(\ref{expr})}{=}\int_0^1\om_{N_0}(\de^l\ph_t,\Ad(\ph_t^{-1})X)dt
=\int_0^1\Big(\int_{N_0}i_{\de^l\ph_t}i_{\ph_t^*X}\mu\Big) dt\\
&=\int_0^1\Big(\int_{N_0}\ph_t^*i_{\de^r\ph_t}i_X\mu\Big) dt
=\int_{\ph_{N_0}}i_X\mu
\end{align*}
is the flux homomorphism \eqref{long}.
\end{proof}

The commutator Lie algebra of $\X(M,\mu)$ is $\X_{ex}(M,\mu)$ \cite{Li74},
so the first continuous cohomology space $H_c^1(\X(M,\mu))
=H_c^1(H^{m-1}(M,\R))=H_{m-1}(M,\R)$.
Under this identification, the flux homomorphism \eqref{long} becomes
\begin{equation}\label{beco}
F_{N_0}([\ph_t])=[\ph_{N_0}]\in H_{m-1}(M,\R).
\end{equation}
Because the first continuous cohomology space of $\X_{ex}(M,\mu)$ is trivial,
the flux homomorphism $F_{N_0}:\pi_1(\Diff_{ex}(M,\mu))\to H_c^1(\X_{ex}(M,\mu))$
vanishes.

The period group $\Ga_\om$ of $\om$ is the image of the period homomorphism 
\begin{equation}\label{peri}
\per_\om:\pi_2(G)\to\R,\quad\per_\om([\si])=\int_{S^2}\si^*\om^{l},
\end{equation}
where $\si$ is a smooth representative for the homotopy class in $\pi_2(G)$.

\begin{theo}\cite{N02}\label{thm2}
Assuming that the period group $\Ga_\om$ is discrete, 
the necessary and sufficient condition for the existence of a Lie group extension of $G$ by $\R/\Ga_\om$, 
integrating the central Lie algebra extension $\R\x_\om\g$,
is the vanishing of the flux homomorphism $F_\om:\pi_1(G)\to H_c ^1(\g)$.
In particular $[\om]$ always integrates to a Lie group extension 
of $\tilde G$ by $\R/{\Ga_\om}$.
\end{theo}

The period homomorphism of the Lichnerowicz cocycle $\om_\et$ is computed in \cite{N04}. 
Given a smooth representative $\si:S^2\to\Diff(M,\mu)_0$ in the homotopy class $[\si]$, to each element $x\in M$ corresponds a 
map $\si_{x}:S^2\to M$.
The integral of the closed 2--form $\et$ over each $\si_x$ provides a function $h_{\si,\et}$ on $M$ with values in the group of periods of $\et$. Now the period homomorphism \eqref{peri} is
\begin{align*}
&\per_{\om_\et}([\si])=\int_Mh_{\si,\et}\mu,
\end{align*} 
so the period group $\Ga_{\om_\et}$ is contained in the period group of the 2-form  $\et$. 

This shows that the period group $\Ga_{\om_\et}$ is discrete if the cohomology class of $\et$ is integral. 
In particular the period group of all cocycles $\om_{N_0}$ is discrete too.
Now the following result concerning the integrability of Lichnerowicz cocycles
follows from lemma \ref{lema} and theorem \ref{thm2}. 

\begin{coro}\label{coro1}
The Lie algebra 2-cocycle $\om_{N_0}$ is integrable to $\Diff_{ex}(M,\mu)$ and to the universal covering group 
$\widetilde\Diff(M,\mu)_0$.
\end{coro}

There are also other coverings of the group of volume preserving diffeomorphisms
where the cocycle $\om_{N_0}$ is integrable. 

\begin{prop}\cite{V06}\label{bar}
Let $\Pi$ be the kernel of the flux homomorphism 
$F_\om:\pi_1(G)\to H_c ^1(\g)$ and let
$\bar G=\tilde G/\Pi$ be the associated covering group of $G$.
Then the central Lie algebra extension $\R\x_\om\g$
integrates to a Lie group extension of $\bar G$ by $\R/\Ga_\om$.
Moreover the covering group $\bar G$ of $G$ is minimal with this property.
\end{prop}

\begin{proof}
The flux homomorphism of $\om$ written for $\bar G$ vanishes, since it
is the restriction of the flux homomorphism
$F_\om$ to $\pi_1(\bar G)=\Pi=\Ker F_\om$.
Knowing that $\pi_2(\bar G)=\pi_2(\tilde G)=\pi_2(G)$, the result follows
from the previous theorem.
\end{proof} 

Knowing the expression \eqref{beco} of the flux homomorphism associated to $\om_{N_0}$,
we obtain the following corollary of proposition \ref{bar}.

\begin{coro}\label{coro2}
A compact codimension 2 submanifold $N_0$ of $M$ being given, 
we consider the subgroup $\Pi_{N_0}$ of $\pi_1(\Diff(M,\mu))$ defined by
$$
\Pi_{N_0}=\{[\ph_t]\in\pi_1(\Diff(M,\mu))|[\ph_{N_0}]=0\in
H_{m-1}(M,\R)\}.
$$
Then the minimal covering group of $\Diff(M,\mu)_0$ on which $\om_{N_0}$ 
can be integrated is 
\begin{equation}\label{bars}
\overline\Diff(M,\mu)_0=\widetilde\Diff(M,\mu)_0/\Pi_{N_0}.
\end{equation}
\end{coro}

  
\section{Non--linear Grassmannians}


The non--linear Grassmannian $\OGr_n(M)$ consists of
all oriented 
compact $n$--dimen\-sional submanifolds without boundary
of a smooth manifold $M$. It is a Fr\'echet manifold in a natural way, see 
\cite{KM97} Section 44.
Suppose $N\in\OGr_n(M)$. Then the tangent space of $\OGr_n(M)$ at $N$ can
naturally be identified with the space of smooth sections of the normal
bundle $TN^\perp:=(TM|_N)/TN$.

The tilda map associates to any $k$--form $\al$ on $M$ a $(k-n)$--form
$\tilde\alpha$ on $\OGr_n(M)$ by:
$$
(\tilde\alpha)_N(Y_1,\dotsc,Y_{k-n})
:=\int_Ni_{Y_{k-n}}\cdots i_{Y_1}\alpha.
$$
Here all $Y_j$ are tangent vectors at $N\in\OGr_n(M)$, \ie sections
of $TN^\perp$. 
This tilda map is related in \cite{V09} to a more general construction, called the hat map,
used to get differential forms on spaces of functions.

There is a natural action of the group $\Diff(M)$ on $\OGr_n(M)$ by
$\ph\cdot N=\ph(N)$. 
For every vector field $X\in\vf(M)$ on $M$, the fundamental vector
field $\zeta_X$ on $\OGr_n(M)$ is $\zeta_X(N)=X|_N$, 
viewed as a section of $TN^\perp$.
One can verify that
\begin{align*}
\widetilde{d\alpha}&=d\tilde\alpha\quad&
i_{\zeta_X}\tilde\alpha=\widetilde{i_X\alpha}\\
L_{\zeta_X}\tilde\alpha&=\widetilde{L_X\alpha}\quad&
\varphi^*\tilde\alpha=\widetilde{\varphi^*\alpha}.
\end{align*}
Theorem 1 in \cite{HV04} shows that if
$[\al]\in H^k(M,\Z)$, then $(\OGr_{k-2}(M),\tilde\al)$ is prequantizable,
\ie there exist a principal $S^1$--bundle
$\mathcal P\to\OGr_{k-2}(M)$ and a principal
connection 1--form $\eta\in\Omega^1(\mathcal P)$ whose curvature form
is $\tilde\alpha$.

Let $p:\widetilde\M\to\M$ denote the universal covering 
projection of the connected component $\M$ of $N_0\in\OGr_n(M)$. 
The elements in $\widetilde\M$ are homotopy classes $[N_t]$ 
of curves $t\mapsto N_t$ of $n$--dimensional submanifolds of $M$, 
starting at $N_0$.
Any closed form $\al\in\Om^{n+1}(M)$ gives rise to a smooth function
$\bar\al$ on $\widetilde\M$, uniquely defined by the conditions
$p^*\tilde\al=d\bar\al$ and $\bar\al([N_0])=0$, with $[N_0]$ denoting the
homotopy class of the constant curve $N_0$.

One can express the function $\bar\al$ by an integral.
We choose a curve $f_t$ of embeddings $N_0\hookrightarrow M$ with $f_t(N_0)=N_t$, 
and we consider an $(n+1)$--chain $c$ in $M$ given by $c:(t,x)\in I\x N_0\mapsto f_t(x)\in M$.
Then $\bar\al([N_t])=\int_c\al$ 
and this integral does not depend on the choice of the embeddings $f_t$, 
so in the sequel we will use the notation 
$$
\bar\al([N_t])=\int_{[N_t]}\al.
$$

Let $M$ be a closed $m$--dimensional manifold with integral volume form 
$\mu$. The codimension 2 non--linear Grassmannian $\OGr_{m-2}(M)$ is a prequantizable
symplectic manifold with symplectic form $\tilde\mu$ \cite{I96}.
On connected components of $\OGr_{m-2}(M)$ the group of exact volume preserving diffeomorphisms acts transitively \cite{HV04}. 
Let $\mathcal M$ be a connected component of $\OGr_{m-2}(M)$ and choose
$N_0\in\mathcal M$. The Lie group
$\Diff_{ex}(M,\mu)$ acts in a Hamiltonian way on 
$(\M,\tilde\mu)$. 
Indeed, 
the fundamental vector field $\ze_{X_\al}$ 
on ${\mathcal M}$ is Hamiltonian with 
Hamiltonian function $\tilde\al$, because $i_{\ze_{X_\al}}\tilde\mu
=i_{\ze_{X_\al}}\mu=\widetilde{i_{X_\al}\mu}=d\tilde\al$, by the tilda calculus. 
The momentum map (non--equivariant in general) is
\begin{equation*}
J:\M\to\X_{ex}(M,\mu)^*,\quad J(N)(X_\al)=\int_N\al-\int_{N_0}\al.
\end{equation*}
The pull--back of the central extension (\ref{kostant}) is a central Lie group extension of $\Diff_{ex}(M,\mu)$.
The corresponding Lie algebra 2--cocycle is $\om_{N_0}$, because
$$
(X,Y)\mapsto -\tilde\mu(\ze_X,\ze_Y)(N_0)=-\int_{N_0}i_Yi_X\mu=\om_{N_0}(X,Y).
$$

\begin{theo}\label{coad}\cite{I96}\cite{HV04}
Let $\mu$ be an integral volume form on $M$ and $N_0$ a codimension 2 submanifold of $M$. Then 
there exists a 1--dimensional central Lie group
extension of $\Diff_{ex}(M,\mu)$, with corresponding Lie algebra extension
of $\X_{ex}(M,\mu)$ defined by the 
2--cocycle $\om_{N_0}$. Moreover $(\mathcal M,\tilde\mu)$ 
is a prequantizable coadjoint orbit of this extension,
with Kirillov--Kostant--Souriau symplectic form.
\end{theo}


\section{Geometric constructions of central Lie group extensions}

The existence of the central Lie group extension of the group of exact volume preserving diffeomorphisms, 
constructed in theorem \ref{coad} via the prequatization extension, 
was already shown in corollary \ref{coro1}.
In this section we present geometric constructions 
of the other Lie group extensions appearing in corollaries \ref{coro1} and \ref{coro2},
namely extensions of coverings of the group of volume preserving diffeomorphisms.

The natural action of $\Diffvol_0$ on the connected component $\mathcal M$ of $\OGr_{m-2}(M)$ 
is not Hamiltonian. By passing 
to universal covering spaces we obtain a Hamiltonian action.
Let $p:\widetilde{\mathcal M}\to\mathcal M$ denote the universal covering space.
The lifted symplectic action of $\widetilde\Diff(M,\mu)_0$ on $\widetilde\M$ is transitive and Hamiltonian. 
The momentum map is
$$
\tilde J:\widetilde{\mathcal M}\to\X(M,\mu)^*,\quad
\tilde J([N_t])(X)=\int_{[N_t]}i_X\mu,
$$
because for $X\in\X(M,\mu)$, the fundamental vector field $\tilde\ze_X$ on $\widetilde{\mathcal M}$ is Hamiltonian with 
Hamiltonian function $\overline{i_X\mu}$ (defined in Section 4): $i_{\tilde\ze_X}p^*\tilde\mu
=p^*\widetilde{i_X\mu}=d(\overline{i_X\mu})$.
It is non--equivariant in general. 

\begin{prop}\label{p4}
The pull--back of the prequantization central extension (\ref{kostant}) associated to 
the prequantizable symplectic manifold $(\widetilde\M,p^*\tilde\mu)$, by the canonical Hamiltonian action of
$\widetilde\Diff(M,\mu)_0$, is a central Lie group extension 
integrating the Lie algebra 2--cocycle $\om_{N_0}$.
\end{prop}

\begin{proof}
The fact that the pull--back is indeed a Lie group, even for infinite dimensional $\M$, follows from Theorem 3.4 in \cite{NV03}.
The Lie algebra cocycle is $\om_{N_0}$ because
$$
-p^*\tilde\mu(\tilde\ze_X,\tilde\ze_Y)([N_0])=
 -\tilde\mu(\ze_X,\ze_Y)(N_0)=-\int_{N_0}i_Yi_X\mu=\om_{N_0}(X,Y),
$$
for all $X,Y\in\X(M,\mu)$.
\end{proof}

A geometric construction of the central extension 
of $\overline\Diff(M,\mu)_0$
can be obtained using the covering space $q:\bar\M\to\M$ defined by
$\bar\M=\widetilde\M/\Pi_\M$ for
$\Pi_\M$ the kernel of the canonical projection
$\pi_1(\mathcal M)\to H_{m-1}(M,\R)$, which
associates to a homotopy class of a loop of $(m-2)$--dimensional submanifolds
the corresponding $(m-1)$--cycle on $M$. This means that $[N_t]\in\Pi_\M$ if and only if $\int_{[N_t]}\be=0$ for all closed $(m-1)$--forms $\be$ on $M$.

\begin{lemm}
The groups $\widetilde\Diff(M,\mu)_0$ and $\overline\Diff(M,\mu)_0$ 
act on $(\bar\M,q^*\tilde\mu)$ in a Hamiltonian way, with momentum map
\begin{equation}\label{muba}
\bar J:\bar\M\to\X(M,\mu)^*,\quad \bar J([N_t])(X)=\int_{[N_t]}i_X\mu.
\end{equation}
\end{lemm}

\begin{proof}
The group $\widetilde\Diff(M,\mu)_0$ acts on 
$\bar\M$ because for any two representing paths
$N_t$ and $N'_t$ of the same element $[N_t]=[N'_t]\in\bar\M$ and for any 
$[\ph_t]\in\widetilde\Diff(M,\mu)_0$, the paths $\ph_t(N_t)$ and $\ph_t(N'_t)$
represent the same element in $\bar\M$.
Indeed, $N_1=N'_1$ and for any closed $(m-1)$--form $\be$ on $M$
\begin{align*}
\int_{[\ph_t(N_t)]}\be=\int_{[N_t]}\be
+\int_{[\ph_t(N_1)]}\be
=\int_{[N'_t]}\be+\int_{[\ph_t(N'_1)]}\be
=\int_{[\ph_t(N'_t)]}\be.
\end{align*}
The action of $\Pi_{\M}\subset \widetilde\Diff(M,\mu)_0$ on $\bar\M$ is trivial.
Indeed, let $[\ph_t]\in\Pi_{\M}$ and $[N_t]\in\bar\M$. Then
for any closed $(m-1)$--form $\be$ on $M$
$$
\int_{[\ph_t(N_t)]}\be-\int_{[N_t]}\be
=\int_{[\ph_t(N_1)]}\be
=\int_{[\ph_t(N_0)]}\be=0,
$$
so $[\ph_t(N_t)]=[N_t]\in\bar\M$.
Finally the $\widetilde\Diff(M,\mu)_0$--action descends to a
$\overline\Diff(M,\mu)_0$--action on $\bar\M$.

Given $X\in\X(M,\mu)$, the fundamental vector field $\bar\ze_X$ on
$\bar\M$ satisfies $Tq.\bar\ze_X=\ze_X$, so
the action is Hamiltonian:
$$
i_{\bar\ze_X}q^*\tilde\mu=q^*i_{\ze_X}\tilde\mu
=q^*\widetilde{i_X\mu}=d(\overline{i_X\mu}),
$$
with Hamiltonian function $\overline{i_X\mu}:[N_t]\mapsto\int_{[N_t]}i_X\mu$, a well defined function on $\bar\M$.
Hence the momentum map is $\bar J([N_t])(X)=\int_{[N_t]}i_X\mu$.
\end{proof}

Now the central extension of $\overline\Diff(M,\mu)_0$, the minimal covering group \eqref{bars} of $\Diff(M,\mu)_0$ on which the Lie algebra cocycle $\om_{N_0}$ can be integrated, can be realized geometrically
with the help of the prequantizable symplectic manifold $\bar\M$.

\begin{prop}
By pulling back the prequantization central extension for $(\bar\M,q^*\tilde\mu)$ by the Hamiltonian $\overline\Diff(M,\mu)_0$--action,
one obtains the central Lie group extension of 
$\overline\Diff(M,\mu)_0$ integrating the cocycle $\om_{N_0}$.
The symplectic manifold $(\bar\M,q^*\tilde\mu)$ 
can be realized as a coadjoint orbit with Kostant-Kirillov-Souriau symplectic form of this central extension 
of $\overline\Diff(M,\om)$, as well as of the central extension of $\widetilde\Diff(M,\om)$ integrating $\om_{N_0}$.
\end{prop}

\begin{proof}
Let $q:\bar\M\to\M$ be the covering defined in the previous section.
Using lemma 1 and proposition 3.4 in \cite{NV03}, and
observing that 
$$
-q^*\tilde\mu(\bar\ze_X,\bar\ze_Y)=\om_{N_0}(X,Y),
$$ 
we get the first part of the proposition.

The actions of $\widetilde\Diff(M,\mu)_0$ and $\overline\Diff(M,\mu)_0$
on $\bar\M$ are Hamiltonian. These are also transitive actions, since 
the $\Diff_{ex}(M,\mu)$--action on $\M$  is transitive \cite{HV04}.
The momentum map is injective because
if $\bar J([N_t])=\bar J([N'_t])$, then $\int_{N_1}\al=\int_{N'_1}\al$
for any $(m-2)$--form $\al$ on $M$ (since $i_X\mu=d\al$ defines
a divergence free vector field $X$). It follows that $N_1=N'_1$ and, since
$\int_{[N_t]}i_X\mu=\int_{[N'_t]}i_X\mu$,
in $\bar\M$ the classes of $[N_t]$ and $[N'_t]$ coincide.

Knowing from proposition 1 in \cite{HV04} that a transitive Hamiltonian $G$--action on 
a symplectic manifold $\M$, with injective momentum map, provides an identification of the symplectic manifold 
with a coadjoint orbit of a 1--dimensional central extension of $G$ with Kostant-Kirillov-Souriau symplectic form, we get the result.
\end{proof}


\end{document}